\documentclass{amsart}
\usepackage[T1]{fontenc} %% permet d'utiliser les caractères accentués
\usepackage{graphicx}

\usepackage{amsmath,amssymb}  %permet d'encadrer formule %Also \begin{align} and many more things
	\usepackage[all]{xy}
	\usepackage{float}
	\usepackage{xcolor}
	\usepackage{tikz-cd}
	\usepackage{todonotes}
	
	\newcommand{\Complex}{\mathbb{C}}
	\newcommand{\Integer}{\mathbb{Z}}
	\newcommand{\CJ}{\mathcal{CJ}}

	\newcommand{\DT}[1]{#1 \dots #1}
	\newcommand{\bydef}{\stackrel{\mbox{\tiny def}}{=}}
	\def \lmod#1\rmod {\left|\smash{#1}\right|\vphantom{#1}}
	
	\newcommand{\pder}[2]{\frac{\partial #1}{\partial #2}}
	\newcommand{\pdertwo}[3]{\frac{\partial^2 #1}{\partial #2 \partial #3}}
	\newcommand{\name}[1]{\operatorname{\mathrm{#1}}}
	\newcommand{\G}{G(m,1,n)}
	\def \lmod#1\rmod {\left|\smash{#1}\right|\vphantom{#1}}

	\theoremstyle{plain}
	\newtheorem{theorem}{Theorem}
	\newtheorem{lemma}[theorem]{Lemma}
	\newtheorem{proposition}[theorem]{Proposition}
	
	\newtheorem{corollary}[theorem]{Corollary}
	\theoremstyle{definition}
	\def \definitionName {Definition}
	\newtheorem{definition}[theorem]{\definitionName}

	\theoremstyle{remark}
	\def \remarkName {Remark}
	\newtheorem{remark}[theorem]{\remarkName}

	\newtheorem{example}[theorem]{Example}

	\numberwithin{theorem}{section}
	\numberwithin{equation}{section}

	\author{Rapha\"el Fesler} 
	\address{R.~F.: Leonard Euler International Mathematical Institute in St.Petersburg, Russia}
	\email{raphael.fesler@gmail.com}
	\author{Denis Gorodkov}
	\address{D.~G.: University of Toronto, Ontario, Canada}
	\email{denis.gorod@gmail.com}
	\author{Maksim Karev}
	\address{M.~K.: Guangdong-Technion Israel Institute of Technology, 241 Daxue Road, Shantou, Guangdong, 515063, China}
	\email{maksim.karev@gtiit.edu.cn}

	\def \Norm {\name{Norm}}

	\DeclareFontFamily{U}{mathx}{\hyphenchar\font45}
	\DeclareFontShape{U}{mathx}{m}{n}{<-> mathx10}{}
	\DeclareSymbolFont{mathx}{U}{mathx}{m}{n}
	\DeclareMathAccent{\widebar}{0}{mathx}{"73}
	
	\title{Hurwitz numbers for reflection groups $G(m,1,n)$}

\begin{document}
		
		\maketitle
		
		\begin{abstract}
			We are extending results from \cite{B-Hurwitz} by building a parallel
			theory of simple Hurwitz numbers for the
			reflection groups $G(m,1,n)$. We also study analogs of the cut-and-join
			operators.  An algebraic description as well as a description
			in terms of ramified covering of Hurwitz numbers is provided. An
			explicit formula for them in terms of Schur polynomials are
			provided.
			In addition the generating function of $G(m,1,n)$-Hurwitz numbers is shown to
			give rise to $m$ independent variables $\tau$-function of the KP hierarchy.
			Finally we provide an ELSV-formula type for these new Hurwitz numbers.
		\end{abstract}

		\tableofcontents
		
		\section*{Introduction}
		Hurwitz numbers were first introduced at the end of the 19th century
		in the work of A. Hurwitz \cite{H91}. Simple Hurwitz numbers 
		$h_{\lambda,m}$ where $\lambda$ is a partition $(\lambda_1,\dots,\lambda_\ell)$
		and $m$ a non-negative integer is the number divided by $n!$ of
		sequences of $m$ transpositions $(\sigma_1 \DT, \sigma_m)$ such that
		their product $\sigma_1 \DT\circ \sigma_m$ belongs to a conjugacy
		class $C_\lambda$ in the symmetric group $S_n$. Hurwitz showed in
		\cite{H91} that equivalently $h_{m,\lambda}$ counts isomorphism classes of ramified
		coverings of the sphere, with $m$ simple critical values, and an extra critical value 
		which has a fixed ramification profile $\lambda$.
		In the last decades, many interests were given to Hurwitz numbers.
		I. Goulden and D. Jackson showed that the generating function of Hurwitz numbers satisfies
		a second order parabolic PDE called the cut-and-join equation \cite{GJ96}.
		A. Okounkov showed in \cite{Okounkov} that the generating function of Hurwitz numbers
		is also a $\tau$-function of the KP hierarchy, implying that
		Hurwitz numbers can be written in terms of Schur polynomials \cite{LandoKazarian}.
		A striking result was proved in a seminal 2001 paper \cite{ELSV01}
		by T.Ekedahl, S.Lando, M.Shapiro and A.Vainshtein (and called the ELSV formula since
		then) built, somewhat unexpected, bridge by revealing deep
		connections of Hurwitz numbers with the geometry of moduli space of
		complex curves. (Namely, the Hurwitz numbers are the intersection
		constants of certain characteristic classes).
		
		Since then many variations of Hurwitz numbers have been introduced \cite{GGPN13, TwistedHurwitz, Lozhkin}, 
		and the most general its version is given by weighted Hurwitz numbers \cite{ACEH}. 
		They satisfy similar results as the ones for simple Hurwitz numbers.
		In this article, we are interested in Hurwitz numbers for the family of reflection
		groups $G(m,1,n)$.
		
		In the last few years, Hurwitz numbers for complex reflection groups have gained 
		more attention, for example in the work of E. Polak and D. Ross \cite{PR20} where polynomiality results are studied,
		the trilogy of papers \cite{DLM22a, DLM22b, DLM23} by T. Douvropoulos, J. B. Lewis, and A. H. Morales where a 
		a genus 0 for connected Hurwitz number for the complex reflection groups is obtained,
		and indirectly in the works of P. Johnson in \cite{Johnson} and H. Zhang and J. Zhou \cite{Zhang}, where a representation
		theory approach was used to study wreath product Hurwitz numbers.
		Our approach, nevertheless, differs from the works quoted above. The present paper
		can be seen as a generalization of the paper \cite{B-Hurwitz} written by the first author,
		applying similar techniques to the groups $G(m,1,n)$, and showing that the theory
		of disconnected simple Hurwitz numbers for the complex reflection groups $G(m,1,n)$ 
		is parallel to the classical one.
		
		In section 1 where we introduce an embedding of $G(m,1,n)$ 
		into $S_{mn}$, the reflections are sent to product of $m$ transpositions and
		power of cycles commuting with the permutation 
		$$\tau\bydef (1\ n+1\ 2n+1\dots(m-1)n+1)\dots(n\ n+n\ 2n+n\ \dots(m-1)n+n).$$
		We then recall the isomorphism between $G(m,1,n)$ and $\Integer/m\Integer\wr S_n$.
		Further, we introduce a convenient description of the conjugacy classes of $G(m,1,n)$
		which are indexed by $m$-tuples of partition $\overline{\lambda}=(\lambda_1,\dots,\lambda_{m-1})$
		(that we call colored partition),
		and show that this description agrees with the classical description of conjugacy classes
		(under a well-chosen isomorphism) of $\Integer/m\Integer\wr S_n$  as in I. Macdonald \cite{Macdonald}.
		
		In section 2 we define Hurwitz numbers for the reflection groups $G(m,1,n)$
		as the product of reflections. It is somewhat similar to the classical
		case where transpositions are used instead.
		
		Section 3 is devoted to the description of the cut-and-join operator and its application.
		We first describe in Theorem \ref{Th:CJ} how multiplying by a reflection affects the conjugacy classes. Then
		using the tools introduced in Section 1 and Section 2, we show in Proposition \ref{Prop:CJinU} that after a well
		chosen change of variable, the cut-and-join for the complex reflection group $G(m,1,n)$
		is actually a direct sum of $m$ independent classical cut-and-join, and $m-1$ sums
		of Euler fields. This allows us to describe the Hurwitz
		numbers of the complex reflection group $G(m,1,n)$ in terms of Schur polynomials (Theorem \ref{Th:G(m,1,n)-Schur}).
		
		In Section 4, Theorem \ref{Th:HGsolKP} we show that the generating function of Hurwitz numbers for 
		the reflection group $G(m,1,n)$ is a $m$-parameter family of $\tau$-functions of
		the KP hierarchy, independently in $m$ variables.
		
		Section 5 uses results from \cite{Zhang}, and describes in terms
		of ramified covering the Hurwitz numbers for reflection group $G(m,1,n)$,
		These are compositions of ramified coverings (Proposition \ref{Prop:HurwRamCover}).
		
		Finally in section 6, Theorem \ref{Th:ELSV} provides an ELSV formula-type for the logarithm of the
		generating function is proved using the simple
		form of the cut-and-join after the change of variables. 
		
		To conclude this introduction, it worth noting that the nice description
		of Hurwitz numbers for the complex reflection group $G(m,1,n)$ in terms
		of Schur polynomials (Theorem \ref{Th:G(m,1,n)-Schur}) motivates to look at
		weighted generalization \cite{GPH}, this will be done in a future work.
		In addition, it is worth applying the tools of this article to the complex
		reflection groups $G(m,p,n)$, it will also be done in some future work.

		\subsection*{Acknowledgements}
		The research of the first
		author was supported by the grant from the Government of
		the Russian Federation, Agreement No.075-15-2019-1620.
		
		\section{Complex reflection groups $G(m,1,n)$}
		\subsection{Definitions and embedding into $S_{mn}$}
		
		Let $V=\Complex^n$ be a complex vector space of dimension $n$ equipped with
		a Hermitian inner product. A \textit{reflection} in $V$ is a unitary automorphism of $V$ of finite order with exactly $n-1$ eigenvalues
		equal to 1. A finite subgroup  $W \subset GL(V)$ is called a \textit{complex reflection group} if it is generated by reflections. 
		
		Let $G(m,1,n)$ be the complex reflection group described in matrix terms as the  $n\times n$ monomial matrices whose non-zero entries are $m^{th}$ roots of unity.
		The reflection elements of this group are:
		\begin{itemize}
			\item reflections $R_{ij}^{(\alpha)}$ of order 2, that fix the hyperplane
			$x_j=\exp(2i\pi/m)^\alpha x_i$ for $1\leq i<j \leq n$ and 
			$\alpha=0,1,\dots,m-1$
			\item reflections $L_i^\alpha$ of order dividing $m$, which fix the hyperplane
			$x_i=0$ for $i=1,\dots,n$ and $\alpha=1,\dots,m-1$
		\end{itemize}
		
		The reflections $\langle L_1^1,R_{1,2}^{(0)},\dots, R_{n,n-1}^{(0)}\rangle$ generate the group $\G$, see \cite{LT09} for details.

		We define now an embedding of $\G$ into $S_{mn}$.
		Let the permutation of order $m$
		\begin{equation}
			\begin{aligned}
				\tau\bydef& (1\ n+1\ 2n+1\dots(m-1)n+1)(2 \ n+2\ 2n+2
				\dots(m-1)n+2)\dots\\
				&(n\ n+n\ 2n+n\ \dots(m-1)n+n)
			\end{aligned}
		\end{equation}
		
		\begin{proposition}\label{Prop:Embedding}
			There exist an embedding $\Psi: G(m,1,n) \hookrightarrow S_{mn}$
			such that
			$\Psi(R_{ij}^{(\alpha)})=(i\ \tau^\alpha(j))(\tau(i)\ \tau^{\alpha+1}(j))\dots
			(\tau^{m-1}(i)\ \tau^{\alpha-1}(j))$ for $\alpha=0,\dots m-1$ and assuming that $\tau^0(j)=j$ and 
			$\Psi(L_i^\alpha)=(i\ \tau(i)\dots \tau^{m-1}(i))^\alpha$, where
			the power of $\alpha$ means multiplication of the permutation
			by itself $\alpha$ times.
			The image of $\Psi$ is the normalizer of $\tau$
			
			\begin{equation*}
				\Psi(G(m,1,n)) = \{\sigma \in S_{mn} \mid \sigma\tau = \tau\sigma\}
				\bydef \Norm(\tau),
			\end{equation*}
		\end{proposition}
		\begin{remark}
			It is also clear that if
			$\sigma \in \Norm(\tau)$ then 
			$\tau^i \sigma(\tau^{-1})^i=\sigma$ for
			$i=1,\dots,m$
		\end{remark}

		\begin{remark}\label{rem:set_action}
			Another way to define this embedding is by considering the faithful action of $G(m,1,n)$ on the $nm$ unit vectors $\mathbf v_{j,k} = (0, \ldots, 0, \exp(2k\pi i/ m), 0, \ldots, 0)$ where the non-zero element is in position $j$, and $k$ runs from $0$ to $m-1$. This faithful action induces an embedding in $S_{mn}$ given the identification $\mathbf v_{j,k} \mapsto j + nk$ with the set $\{1, \ldots, mn\}$ on which $S_{mn}$ is acting.
		\end{remark}

		To prove Proposition \ref{Prop:Embedding} we describe the
		cycle decomposition of elements in $\Norm(\tau)$.
		
		\noindent Introduce a notation more convenient for our purposes: 
		
		\begin{equation}\label{Eq:rij}
			r^{(\alpha)}_{ij}
			\bydef (i\ \tau^\alpha(j))(\tau(i)\ \tau^{\alpha+1}(j))\dots
			(\tau^{m-1}(i)\ \tau^{\alpha-1}(j)) 
		\end{equation}  
		for all $1 \le i \ne j \le mn, i \ne j\pm kn$ for $ k=1,\dots,m-1$
		addition is modulo $mn$. So if $1 \le i < j \le n$ then $r_{ij}^{(\alpha)} =r_{\tau(i),\tau(j)}^{(\alpha)}=\dots =
		r_{\tau^{m-1}(i),\tau^{m-1}(j)}^{(\alpha)} = \Phi(R_{ij}^{(\alpha)})$. 
		
		Also denote 
		\begin{equation}\label{Eq:li}
			l_i^\alpha \bydef (i\ \tau(i)\dots \tau^{m-1}(i))^\alpha\ \text{for all}\ 1 \le i \le
			mn;
		\end{equation}
		so, $l_i^{\alpha} = l_{\tau(i)}^{\alpha}=\dots=
		l_{\tau^{m-1}(i)}^{\alpha}= \Phi(L_i^\alpha)$ if $i \le n$.

		In the following proposition, the power of
		$\tau$ are considered modulo $m$, denote also the integers
		$k_\alpha=\frac{m}{\gcd(m,\alpha)}$.
		
		\begin{proposition}\label{Prop:CycleDecomp}
			
			Let $x\in \Norm(\tau)$, and
			let $x=a_1\dots a_r$ be its
			cycle decomposition. 
			Then for every $a_j$, $j = 1\dots r$ one
			of the following is true:
			\begin{enumerate}
				\item the cycle decomposition contains
				$m-1$ other cycles $a_{n_1}, \ldots, a_{n_{m-1}}$ of the same length such that $$a_j = \tau a_{n_1}\tau^{-1}= 
				\tau^2a_{n_2}\tau^{-2}=\dots =\tau^{m-1}a_{n_{m-1}}\tau^{-(m-1)},$$ or
				
				\item for some $1\leq \alpha\leq m-1$, the cycle
				$a_j$ has the following structure: 
				its length is a multiple of $k_\alpha$, it is $\tau^\alpha$ invariant, and
				there are $\gcd(m,\alpha)-1$ other cycles $a_{j_{2}},\dots,a_{j_{\gcd(m,\alpha)}}$
				having the same length, such that $a_j = \tau a_{j_2}\tau^{-1}= 
				\tau^2a_{j_3}\tau^{-2}=\dots =\tau^{\gcd(m,\alpha)-1}a_{{\gcd(m,\alpha)}}\tau^{-(\gcd(m,\alpha)-1)}.$	 	
			\end{enumerate}
		\end{proposition}
		
		The first case will be called a \textit{$\beta_0$-cycle}, and the second
		\textit{$\beta$-cycles of type $\alpha$} and denoted $\beta_\alpha$-cycle
		for short. The $\beta_0$-cycle will play a special role later on in the 
		text, that is why they are extracted from the case (2), where they
		can be seen as a $\beta_\alpha-$cycle with $\alpha=0$.
		
		\begin{example}
			In $\Integer/6\Integer\wr S_n$ 
			$$\tau = (1\ n+1\dots 5n+1)\dots(n\ 2n\dots 6n)$$
			and let $a\bydef (a_1\dots a_k)$ such that
			$a_i\neq \tau^\alpha(a_j)$ for $i,j=1,\dots,k$ and $\alpha=0,\dots ,5$
			then we have the following possible $\beta_\alpha-$cycles
			\begin{itemize}
				\item $\beta_0-$cycle: $(a)(\tau(a))\dots(\tau^5(a))$ 
				\item  $\beta_1$-cycle: $(a,\tau(a),\tau^2(a),\dots,\tau^5(a))$ 
				\item  $\beta_2$-cycle: $(a,\tau^2(a),\tau^4(a))(\tau(a),\tau^3(a),\tau^5(a))$ 
				\item $\beta_3-$cycle: $(a,\tau^3(a))(\tau(a),\tau^4(a))(\tau^2(a),\tau^5(a))$ 
				\item  $\beta_4$-cycle: $(a,\tau^4(a),\tau^2(a))(\tau(a),\tau^5(a),\tau^3(a))$ 
				\item  $\beta_5$-cycle: $(a,\tau^5(a),\tau^4(a),\dots,\tau(a))$ 
			\end{itemize}
			
			where for example in the $\beta_2$-cycle, we have that $k_\alpha=k_2=3$.
		\end{example}
		
		\begin{proof}
			To ease the notation let 
			\begin{equation*}
				x=a_1a_2\dots a_r
			\end{equation*}
			be the cycle of decomposition of $x$, where 
			\begin{equation*}
				a_i = (a_{1i}\dots a_{1m_i}).
			\end{equation*}
			
			Given that $\tau^\alpha x \tau^{-\alpha}=x$ for $\alpha=1,\dots,m$
			by the definition of $\Norm(\tau)$, this implies 
			\begin{equation}\label{Eq:CondTau}
				\begin{aligned}
					&x=\tau(a_1)\dots \tau(a_{r})\\
					&x=\tau^2(a_1)\dots \tau^2(a_{r})\\
					&\vdots\\
					&x=\tau^{m-1}(a_1)\dots \tau^{m-1}(a_{r}).
				\end{aligned}
			\end{equation}
			By the uniqueness of the cycle decomposition, every $\tau(a_j),
			\tau^2(a_j),\dots,\tau^{m-1}(a_j)$
			must be equal to some $a_s$.
			If all these cycles are different they form a $\beta_0$ cycle.
			Let now	$1\leq \alpha<m$, then 
			$\tau^\alpha(a_s) = (\tau^\alpha(a_{1s})\dots \tau^\alpha(a_{1m_s}))$
			and one has that  $\tau^\alpha(a_s)\neq \tau^{2\alpha}(a_s)\neq\dots
			\neq\tau^{k_\alpha}(a_s)$.
			By definition of $k_\alpha$ we have $\tau^\alpha(\tau^{k_\alpha}(a_s))=a_s$
			as $k_\alpha+\alpha =0\mod m$, implying that we have a cycle of the form 
			$a = (a_s,\tau^\alpha(a_s),\tau^{2\alpha}(a_s),\dots, \tau^{k_\alpha}(a_s))$. 
			Furthermore the condition
			given in (\ref{Eq:CondTau}), and the uniqueness of cycle decomposition
			force
			to have $\gcd(m,\alpha)-1$ other cycles $a_{j_{2}},\dots,a_{j_{\gcd(m,\alpha)}}$
			having the same length, such that $a = \tau a_{j_2}\tau^{-1}= 
			\tau^2a_{j_3}\tau^{-2}=\dots =\tau^{\gcd(m,\alpha)-1}a_{{\gcd(m,\alpha)}}\tau^{-(\gcd(m,\alpha)-1)}$	 	
			and this finishes the proof.
			
		\end{proof}

		\begin{proof}[Proof of proposition \ref{Prop:Embedding}]
			By remark \ref{rem:set_action}, the map
			$\Psi: \G\mapsto S_{mn}$, is a group
			homomorphism and an embedding. We write
			$\Psi(x)=\tilde d_1,\DT \tilde d_N$, where $\tilde{d}_i$ is
			the the image of the reflection $d_i$ under $\Psi$.
			
			The elements $\Psi(r^{(\alpha)}_{ij})$ and $\Psi(\ell^\gamma_i)$ 
			commute with $\tau$, implying that $\Psi(\G)\subset \Norm(\tau)$.
			
			We are left to show that $ \Norm(\tau)\subset \Psi(\G)$.
			To do so, we prove that any $x\in \Norm(\tau)$, which is a product
			of $\beta_0-$cycle and $\beta_\alpha$-cycles, can be written
			in terms of product of $r^{(k)}_{ij}$ and $\ell^k_i$.
			
			For a $\beta_0-$cycle
			\begin{equation*}
				(a_1\dots a_j)(\tau(a_1)\dots\tau(a_j))\dots
				(\tau^{m-1}(a_1)\dots\tau^{m-1}(a_j))=r_{a_{j-1},a_j}^{(0)}\dots r_{a_2,a_3}^{(0)}r_{a_1,a_2}^{(0)}\
			\end{equation*} 
			
			Now, a straightforward computation shows that
			if $x\in \beta_\alpha$ then $x\ell_\gamma\in \beta_{\alpha+\gamma}$ 
			(see also section \ref{MultByl}). It follows that
			for a $\beta_\alpha$-cycle  $a = (a_{1}\dots a_{j})$
			
			\begin{equation*}
				\begin{aligned}
					&(a,\tau^\alpha(a),\tau^{2\alpha}(a),\dots,\tau^{k_\alpha}(a))(\tau(a),\tau^{\alpha+1}(a),
					\tau^{2\alpha+1}(a),\dots,\tau^{k_{\alpha+1}}(a))\dots\\
					&\dots (\tau^j(a),\tau^{\alpha+\gamma}(a),
					\tau^{2\alpha+\gamma}(a),\dots,\tau^{k_{\alpha+\gamma}}(a))\\
					&=\ell_j^{\gamma}r_{a_{j-1},a_j}^{(0)}\dots r_{a_2,a_3}^{(0)}r_{a_1,a_2}^{(0)}
				\end{aligned}
			\end{equation*}
			as $r_{a_{j-1},a_j}^{(0)}\dots r_{a_2,a_3}^{(0)}r_{a_1,a_2}^{(0)} \in \beta_0$.
		\end{proof}
		
		From now on, and by abuse of notation, we denote by $\G$ both the image
		$\Psi(\G)\subset S_{mn}$ and $\G$ itself.

		\subsection{Wreath product, correspondence and conjugacy classes}
		
		There is a way to represent the elements of $G(m,1,n)$ \cite{ST54},\cite{DLM22a}-\cite{DLM23}. Namely,
		given an element $z$ in $G(m,1,n)$, one can encode it
		by a pair $[u;g]$ with $u\in S_n$ and 
		$g=(g_1,\dots,g_n)\in (\Integer/ m\Integer)^n$ as follows:
		for $k= 1,\dots, n$ the non-zero entry in column $k$ of $z$ is in row 
		$u(k)$, and the value of the entry is $\exp(\frac{2i\pi g_k}{m})$.
		The product rule reads 
		
		\begin{equation*}
			[u;g]\cdot[v;b]= [uv;v(g) +b] \ \text{where} \ v(g)=(g_{v(1)},\dots,g_{v(n)})
		\end{equation*}
		
		This description provides $G(m,1,n)$ with a structure of the wreath product 
		$\Integer/m\Integer \wr S_n$. The reflections
		$r_{ij}^{(\alpha)}$ and $\ell_i^\alpha$ are represented as follows: 
		\begin{gather*}
			r_{ij}^{(\alpha)} =  [(i,j);0,\dots,0,\alpha,0,\dots,0,-\alpha,0,\dots,0],\\
			\ell_i^\alpha = [(id);0,\dots,0,\alpha,0,\dots,0].\end{gather*}
		
		It is convenient to introduce an action of $\mathbb{Z}/m\mathbb{Z} \wr S_n $ on the set $\{1,\ldots,mn\}$  %inducing the embedding in $S_{mn}$ coinciding with the embedding of $G(m,1,n)$ under the contructed isomorphism 
		using the figure \ref{fig:circles}. The element $[\sigma,(g_1, \ldots, g_n)]$ acts by permuting the circles in accordance with $\sigma$, and the $i$-th circle is rotated by $g_i$ counterclockwise (which corresponds to adding $g_i$ in $\mathbb{Z}/m\mathbb{Z}$).This induces a faithful action, and, therefore, induces an embedding $\mathbb{Z}/m\mathbb{Z} \wr S_n \hookrightarrow S_{mn}$ (compare with remark \ref{rem:set_action}).
		
		Elements of $\mathbb{Z}/m\mathbb{Z} \wr S_n$ correspond to permutations on the $mn$ elements that preserve the circles (possibly permuting them entirely) with the order of the elements in each individual circle. The reflection $r_{ij}^{(\alpha)}$ would correspond to rotating the $i$-th circle by $\alpha$, rotating the $j$-th circle by $-\alpha$, and permuting them, whereas the reflection $l_i^\alpha$ just rotates the $i$-th circle by $\alpha$.
		\begin{figure}
			\centering
			\includegraphics[scale=0.65]{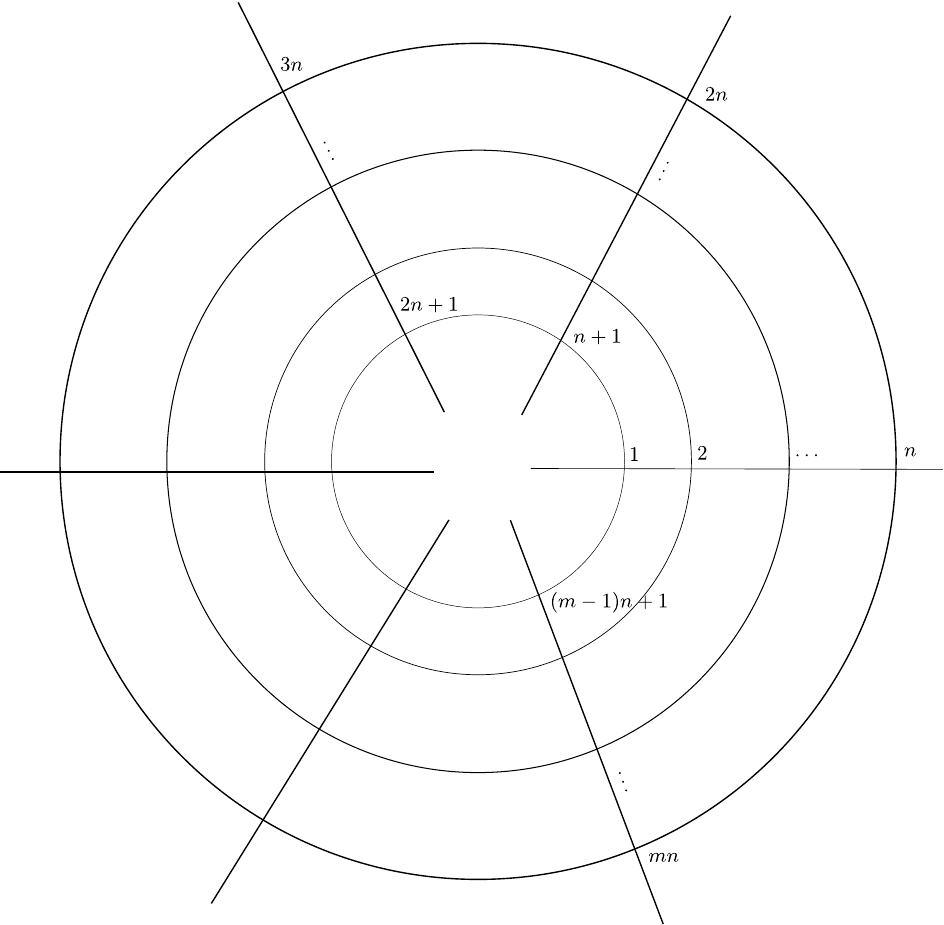}
			\caption{Presentation of the set of $mn$ elements}
			\label{fig:circles}
		\end{figure}
		
		The conjugacy classes of the symmetric group $S_{n}$ are in one-to-one correspondence with partitions of $n$, provided by cycle decomposition of the elements. A similar conjugacy class description for $\Integer/m\Integer \wr S_n$
		can be found in \cite{Macdonald}.
		We remind it here for completeness following closely \cite{Macdonald},
		then translate it into the language developed in the previous section.
		
		Denote $\Lambda$ to be the set of all partitions $\Lambda = \cup_{n \ge 0}\{\lambda\ |\ \lambda \vdash n\}$. A \emph{$m-$colored partition} is a map $\mathbb Z / m\mathbb Z \to \Lambda$.
		A pair $p =[u;g]$ with $u\in S_n$ and 
		$g=(g_1,\dots,g_n)\in (\Integer/ m\Integer)^n$ defines an $m-$colored partition as follows: the image of the element $r\in \mathbb Z /m \mathbb Z$ is a partition formed by lengths of those cycles $(i_1\ldots i_k)$ of $u\in S_n$, such that their \emph{cycle-product} $g_{i_1} + \cdots g_{i_k}$ equals $r$. The constructed $m-$colored partition determines the conjugacy class of $[u,g]$, and two pairs are of the same conjugacy classes if the corresponding colored partitions coincide.
		
		We adopt the following notation for  $m-$colored partitions: for a $m-$tuple of partitions $\lambda_0,\ldots,\lambda_{m-1}\in \Lambda$ write $\bar \lambda = \lambda_0|\ldots|\lambda_{m-1}$ for the map 
		$$\bar \lambda \colon \begin{matrix}\mathbb Z/m \mathbb Z \to \Lambda \\ r \mapsto \lambda_r\end{matrix}.$$
		
		%Fix  an $m-$colored partition $\overline{\lambda}$ such that
		%		$|\overline{\lambda}|=
		%		|\lambda_0|+\dots+|\lambda_{m-1}|=\lambda_{0,1}+\dots\lambda_{0,k_1}
		%		+\dots+\lambda_{m-1,k_{m-1}}=n$. For the wreath product $\Integer/m\Integer \wr S_n$
		%		the conjugacy classes are defined as follows:
		%		for a pair $p =[u;g]$ with $u\in S_n$ and 
		%		$g=(g_1,\dots,g_n)\in (\Integer/ m\Integer)^n$, the permutation
		%		$u$ can be written as a disjoint union of cycles. 
		%For each
		%		cycle $\sigma = (i_1\;\dots\;i_k)$ the sum $g_{i_i}+\dots +g_{i_k}$
		%		is an element of $\Integer/m\Integer$ 
		%		and is determined by $\sigma$. It is called the \textit{cycle-product} of
		%		$p$ corresponding to the cycle $\sigma$.
		%		For each element $g\in \Integer/ m\Integer$ and each integer $k\geq 1$
		%		let $m_k(g)$ denote the number of $k-$cycles in $u$ whose cycle-product
		%		lies in $g$. In this way each element of $\Integer/m\Integer \wr S_n$
		%		determines an array $(m_k(g))_{\substack{k\geq1 \\ g\in \Integer/m\Integer}}$
		%		of non negative integers such that $\sum_{g,k}km_k(g)=n$, which is equivalent to
		%		choosing a colored partition of $[u;g]$.
		%		Two elements of $\Integer/m\Integer \wr S_n$ are conjugate if and only if
		%		they have the same colored partition.
		
		We now show that the sets 
		$C_{\overline{\lambda}}=C_{\lambda_0|\dots|\lambda_{m-1}}$ 
		of elements in $G(m,1,n)\subset S_{mn}$ 
		such that their cycle decomposition contains $\beta_\alpha$-cycles with cycle structure $\lambda_\alpha$ for $\alpha=0,\dots,m-1$ are the conjugacy classes of  $\Integer/m\Integer \wr S_n$.
		
		We will show this using the previously discussed embedding of $\mathbb{Z}/m\mathbb{Z}\wr S_n$ into $S_{mn}$. Explicitly, it can be described as follows:
		
		\begin{multline*}
			[(g_1, \ldots, g_n),
			\bigl(
			\begin{smallmatrix}
				1 & 2 & 3 & \ldots & n \\
				\sigma(1) & \sigma(2) & \sigma(3) & \ldots & \sigma(n)
			\end{smallmatrix}
			\bigr)] \longmapsto \\
			\longmapsto
			\bigl(
			\begin{smallmatrix}
				1 & 2 & 3 & \ldots & n & \ldots & kn+1 & \ldots\ldots & (k+1)n & \ldots \\
				\sigma(1) + g_1 n & \sigma(2) + g_2 n & \sigma(3) + g_3 n & \ldots & \sigma(n) + g_n n & \ldots\ldots & (\sigma(1) + kn) + g_1n & \ldots & (\sigma(n) + kn) + g_n n & \ldots
			\end{smallmatrix}
			\bigr)
		\end{multline*}
		
		where the images of the permutation in $S_{mn}$ are considered modulo $mn$.
		
		Then, let us compute the image of the element $ [(g_1, \ldots, g_n), (i_1\; i_2\; \ldots \; i_k)] \in \mathbb{Z}/m\mathbb{Z}\wr S_n $ in $S_{mn}$ for an individual cycle $\sigma = (i_1 \; i_2 \; \ldots \; i_k)$ in $S_n$. If we try to recover the cycle structure, we will get the following:
		\begin{align*}
			i_1 &\mapsto i_2 + g_{i_1}n \mod mn\\
			i_2 + g_{i_1}n \mod mn&\mapsto (i_3 + g_{i_2}n) + g_{i_1}n \mod mn\\
			&\ldots\\
			i_k + (g_{i_1} + g_{i_2} + \ldots + g_{i_k-1})n \mod mn&\mapsto i_1 + (g_{i_1} + a_{i_2} + \ldots + g_{i_k})n \mod mn
		\end{align*}
		This proceeds until, after repeating this process some $l$ times, $l(g_{i_1} + g_{i_2} + \ldots + g_{i_k})$ becomes divisible by $m$, and we return back to $g_1$ completing the cycle. If $l = m$, then the image is this one cycle, and this corresponds to the $\beta_0$ cycles in notations from the previous section. However, if $l|m$, then the image will contain $m/l$ cycles received from the constructed one by adding $n$ to all the elements. The resulting permutation type depends on $k$ and $(g_{i_1} + g_{i_2} + \ldots + g_{i_k}) \mod m$, which is exactly the input from the colored partition. Moreover, $g_1+\dots+ g_k \mod m$ is exactly the power $\alpha$ of $\tau$ which defines the $\beta_\alpha$-cycle.
		
		This finishes the proof of
		\begin{proposition}\label{Prop:ConjClass}
			The set $C_{\overline{\lambda}}$ is a conjugacy class.
			Every conjugacy class in $\G$ is $C_{\overline{\lambda}}$
			for some partitions $\lambda_0,\dots,\lambda_{m-1}$ such that 
			$|\lambda_0|+\dots+|\lambda_{m-1}|=n$
		\end{proposition}
		
		In particular,  the reflections $r_{i,j}^{(k)}$ form the conjugacy
		class $C_{2,1^{n-2}|\emptyset|\dots|\emptyset}$ and the reflections
		$\ell_i^k$ the conjugacy classes 
		$C_{1^{n-1}|\emptyset|\dots|1|\emptyset|\dots}$
		where 1 is in position $k$.
		
		We denote by $\mathcal{R}$ the conjugacy class 
		$C_{2,1^{n-2}|\emptyset|\dots|\emptyset}$ and by
		$\mathcal{L}^k$ the conjugacy class 
		$C_{1^{n-1}|\emptyset|\dots|1|\emptyset|\dots}$.

		\section{Hurwitz numbers}
		From now on until the end of the text, the integer $n$ and $m$ are fixed.
		Fix now the colored partition $\overline{\lambda}$ such that
		$|\lambda_0|+\dots+|\lambda_{m-1}|=n$, and let the 
		conjugacy classes $\mathcal{R}$ and $\mathcal{L}^k$ for
		$k=1,\dots,m-1$ as defined above.

		\begin{definition}\label{Df:HurwNumb}
			Denote by $M\bydef n_0+n_1+\dots+n_{m-1}$.
			A sequence of reflections $(\sigma_1 \DT, \sigma_{M})$
			of the group $\G$ is said to have {\em profile} 
			$(\overline{\lambda},n_0,n_1,\dots,n_{m-1})$ if
			$\#\{p \mid \sigma_p \in {\mathcal R}\} = n_0$, $\#\{p \mid
			\sigma_p\in {\mathcal L}^1\} = n_1,\dots,
			\#\{p \mid\sigma_p\in {\mathcal L}^{m-1}\} = n_{m-1}$ and $\sigma_1 \dots
			\sigma_{M} \in C_{\overline{\lambda}}$. 
			The {\em Hurwitz numbers
				for the group $\G$} are $h_{n_0,n_1,\dots,n_{m-1}\overline{\lambda}}
			= \frac{1}{m^nn!}\#\{(\sigma_1 \DT, \sigma_M) \text{ is a sequence of profile
			} (\overline{\lambda},n_0,n_1,\dots,n_{m-1})\}$.
		\end{definition}
		
		Denote  $\mathcal{C}_{\overline{\lambda}} \bydef \frac{1}{\#C_{\overline{\lambda}}}
		\sum_{x \in C_{\overline{\lambda}}} x \in \Complex[\G]$. As, by proposition~\ref{Prop:ConjClass} the conjugacy classes are sets $C_{\overline{\lambda}}$ for various $\overline{\lambda}$, these elements form a basis in the center of the group algebra $Z[\G]$. %One has $\mathcal{C}_{\overline{\lambda}} \in Z[\G]$ %(the
		%center of the group algebra of $\G$); %by Proposition
		%, $\mathcal{C}_{\overline{\lambda}}$ form a basis in
		%$Z[\G]$.
		Consider now a ring of polynomials 
		$\Complex[\mathbf{p}]\bydef\Complex[p^{(0)},p^{(1)},\dots,p^{(m-1)}]$
		where $p^{(i)} = (\frac{p^{(i)}_1}{m}, \frac{p^{(i)}_2}{m^2}, \dots)$\footnote{These extra factors 
			$\frac{1}{m^i}$ which seems unnecessary are present to get nicer formula, especially in corollary \ref{Cor:exp(p1)}.},
		(by abuse of notation we drop the factors in front of the variables and
		write $(p^{(i)}_1, {p^{(i)}_2}, \dots)$ instead) are $m$
		infinite sets of variables. The ring is graded by the total degree
		where one assumes $\deg p^{(i)}_k =k$ for all $k=1,2,\dots$. The linear
		map $\Theta$ defined by
		\begin{equation}\label{Eq:DefIsoB}
			\Theta(\mathcal{C}_{\overline{\lambda}}) = \mathbf{p}_{\overline{\lambda}} 
			\bydef p^{(0)}_{\lambda_0}p^{(1)}_{\lambda_1},\dots p^{(m-1)}_{\lambda_{m-1}}
		\end{equation}
		establishes an isomorphism between
		the vector spaces $Z[\G]$ and the
		homogeneous component $\Complex[p^{(0)},p^{(1)},\dots,p^{(m-1)}]_n$ of total degree $n$.
		
		Now denote by
		\begin{equation}\label{Eq:DefT0}
			\mathcal T_0 := \frac1m \sum_{\substack{1 \le i < j \le mn\\
					0\leq k\leq m-1} }
			r_{ij}^{(k)} = \#\mathcal{R} \cdot \mathcal{C}_{1^{n-2}2|\emptyset|\dots\emptyset}
		\end{equation}
		and for all $k=1,\dots,m-1$
		\begin{equation*}
			\mathcal T_k := \frac1m \sum_{1 \le i \le mn} l_i^k = 
			\#\mathcal{L}^k
			\cdot \mathcal{C}_{{1^{n-1}}|\emptyset|\dots|1|\emptyset|\dots|\emptyset}
		\end{equation*}
		sums of all elements of the conjugacy classes containing reflections, see equations (\ref{Eq:rij}) and (\ref{Eq:li}) for the definitions of $r^{(\alpha)}_{ij}$ and $\ell_i^k$.
		Now $r^{(k)}_{ij} = r^{(k)}_{\tau(i),\tau(j)}=\dots=r^{(k)}_{\tau^{m-1}(i),\tau^{m-1}(j)}$,
		and $l_i^k = l_{\tau(i)}^k=\dots=l_{\tau^{m-1}(i)}^k$, so every
		element is repeated $m$ times in these sums; hence the factor
		$\frac1m$). The elements $\mathcal T_0$ and $\mathcal T_k's$ belong to
		$Z[\G]$, so one can consider linear operators $T_0, T_1,\dots,T_{m-1}: Z[\G] \to
		Z[\G]$ of multiplication by $\mathcal T_0,\dots,\mathcal T_{m-1}$,
		respectively. The operators $T_i$ and $T_j$ pairwise commute for all
		$i,j = 0,\dots,m-1$.
		
		Let now the family of $m$ operators $\CJ_0,\dots,\CJ_{m-1}$ the operators
		making the following diagrams commutative:
		
		\begin{equation}\label{Eq:CJViaGAlg}
			\xymatrix{
				Z[\G] \ar[r]^{\times T_i} \ar[d]^\simeq &
				Z[\G] \ar[d]^\simeq  \\
				\mathbb{C}[\mathbf{p}]_n \ar[r]_{\CJ_i} &
				\mathbb{C}[\mathbf{p}]_n   
			}, \qquad i = 0,\dots,m-1.
		\end{equation}
		
		Let the two colored partition $\overline{\lambda}$ and
		$\overline{\mu}$ such that $|\overline{\lambda}|=
		|\overline{\mu}|=n$. For an element $\sigma\in 
		C_{\overline{\lambda}}$,
		define the multiplicity $ \langle\overline{\lambda}|
		\overline{\mu}\rangle_0$ as the number of reflections
		$r_{i,j}^{(k)}$ (i.e elements of $\mathcal{R}$) such that 
		$\sigma r_{i,j}^{(k)}\in C_{\overline{\mu}}$.
		Similarly define the multiplicity 
		$ \langle\overline{\lambda}|
		\overline{\mu}\rangle_k$ for elements in 
		$\mathcal{L}^k$ and $k=1,\dots,m-1$
		
		The following lemma and theorem are proved
		in the exact same way as in \cite{B-Hurwitz}, we 
		refer to this paper for more details.
		
		\begin{lemma}\label{Lm:Conj}
			Multiplicities do not depend on the choice of $\sigma$.
		\end{lemma}
		
		\begin{theorem}\label{Th:TViaMult}
			$T_i\mathcal{C}_{\overline{\lambda}} = 
			\sum_{\overline{\mu}} \langle \overline{\lambda}
			\mid \overline{\mu}\rangle_i\cdot \mathcal{C}_
			{\overline{\mu}}$ for $i = 0,\dots,m-1$.
		\end{theorem}
		\section{Generating function, cut-and-join and explicit formulas}
		\subsection{Generating function} 
		Consider the following generating function for Hurwitz numbers of the
		group $\G$ :
		\begin{equation*}
			\mathcal{H}(\beta_0,\beta_1,\dots,\beta_{m-1},\mathbf{p}) =
			\sum_{n_0,\dots,n_{m-1}}\sum_{\overline{\lambda}}
			\frac{ h_{n_0,\dots,n_{m-1},\overline{\lambda}}}{n_0!n_1!\dots n_{m-1}!}
			p_{\overline{\lambda}}\beta_0^{n_0}\dots\beta^{n_{m-1}}_{m-1}.
		\end{equation*}
		\begin{theorem}\label{Th:CJequ}
			The generating function $\mathcal{H}$ satisfies
			the family of $m$ cut-and-join equations:
			
			\begin{equation}\label{Eq:CJ}
				\pder{\mathcal{H}}{\beta_0} = \CJ_0(\mathcal{H}) \quad ,\dots,
				\quad \pder{\mathcal{H}}{\beta_{m-1}} = \CJ_{m-1}(\mathcal{H})
			\end{equation}
		\end{theorem}
		
		\begin{proof}
			Fix a positive integer $n$ and denote by $\mathcal{H}_n$ a degree
			$n$ homogeneous component of $\mathcal{H}$. The cut-and-join
			operators preserve the degree, so $\mathcal{H}$ satisfies the
			cut-and-join equations if and only if $\mathcal{H}_n$ does (for each
			$n$).
			
			Let $i$ any integer between $0$ and $m-1$, and let
			\begin{equation*}
				\mathcal{G}_n \bydef \sum_{n_0,\dots n_{m-1} \ge 0}
				\sum_{\overline{\lambda}:\mid\overline{\lambda}\mid = n}
				\frac{m^nn!h_{n_0,\dots n_{m-1},\overline{\lambda}}}
				{n_0!\dots n_{m-1}!} \mathcal{C}_{\overline{\lambda}}
				\beta_0^{n_0} \beta_1^{n_1}\dots\beta_{m-1}^{n_{m-1}} \in\Complex[\G]
			\end{equation*}
			An elementary combinatorial reasoning gives
			\begin{equation*}
				\mathcal{G}_n = \sum_{m\geq 0}
				\frac{\beta_0^{n_0}\beta_1^{n_1}\dots\beta_{m-1}^{n_{m-1}}}
				{n_0!\dots n_{m-1}!}T_0^{n_0}\dots T_{m-1}^{n_{m-1}}(e_n)
			\end{equation*}
			where $e_n \in \G$ is the unit element. Clearly
			\begin{equation*}
				\begin{aligned}
					T_i(\mathcal{G}_n) = \sum_{n_0,\dots n_{m-1} \ge 0}
					\frac{\beta_0^{n_0}\beta_1^{n_1}\dots\beta_{m-1}^{n_{m-1}}}
					{n_0!\dots n_i!\dots n_{m-1}!}T_0^{n_0}\dots(T_i)^{n_i+1}
					\dots T_{m-1}^{n_{m-1}}(e_{n}) =\\ 
					\sum_{\substack{n_0\dots \hat n_i,\dots n_{m-1} \ge
							0\\ n_i\geq 1}} \frac{\beta_0^{n_0}\dots \beta_{i}^{n_i-1}\dots \beta_{{m-1}}^{n_{m-1}}}
					{n_0!\dots (n_i-1)!\dots n_{m-1}} T_0^{n_0}\dots T_i^{n_i}
					\dots T_{m-1}^{n_{m-1}}(e_{n}) =
					\pder{\mathcal{G}_n}{\beta_\alpha}.
				\end{aligned}
			\end{equation*}
			Applying the isomorphism $\Psi$ one
			obtains $\Psi T_i(\mathcal{G}_n) = \Psi(\pder{\mathcal{G}_n}{\beta_\alpha})
			= \pder{}{\beta_\alpha} \Psi(\mathcal{G}_n)$.  One has $\Psi(\mathcal{G}_n)
			= \mathcal{H}_n$, hence $\pder{}{\beta_\alpha} \Psi(\mathcal{G}_n) =
			\pder{\mathcal{H}_n}{\beta_\alpha}$.
			By the
			definition of the cut-and-join operators, $\Psi T_i(\mathcal{G}_n) =
			\CJ_i(\Psi(\mathcal{G}_n)) = \CJ_i({\mathcal{H}_n})$, and
			equalities \eqref{Eq:CJ} follow.
		\end{proof}
		\begin{corollary}\label{Cor:exp(p1)}
			\begin{equation}\label{Eq:GenFunc}
				\mathcal{H}(\beta_0,\dots,\beta_{m-1},\mathbf{p}) =
				e^{\beta_0\CJ_0+\dots+\beta_{m-1}\CJ_{m-1}}e^{p^{(0)}_1}
			\end{equation}
		\end{corollary}
		
		\begin{proof}
			It follows from Definition \ref{Df:HurwNumb} that
			$h_{0,\dots,0,\overline{\lambda}} = \frac{1}{m^nn!}$ if $\lambda_0 = 1^n$ and
			$\lambda_i=\emptyset$ for $i= 1,\dots,m-1$, and $h_{0,\dots,0,\overline{\lambda}} = 0$ otherwise.  Thus
			$\mathcal{H}(0,\dots,0,\mathbf{p}) = e^{p^{(0)}_1}$ and \eqref{Eq:GenFunc} follows
			from Theorem \ref{Th:CJequ}.
		\end{proof}
		
		\subsection{Explicit formulas} We provide in this section explicit formulas
		for the cut-and-join and the generating function.
		
		\begin{theorem}\label{Th:CJ}
			The $m$-family of cut-and-join operators for the reflection groups $\G$ is 
			given by
			\begin{align}\label{Eq:CJ_0}
				\mathcal{CJ}_0 =\frac{1}{2} \sum_{\substack{i,j \ge 1\\ 
						\alpha,\gamma \in \Integer/m\Integer }}
				(i+j) p^{(\alpha)}_i p^{(\gamma)}_j \pder{}{p^{(\alpha+\gamma)}_{i+j}}
				+mij p_{i+j}^{(\alpha+\gamma)} 
				\frac{\partial^2}{\partial p^{(\alpha)}_ip^{(\gamma)}_j}
			\end{align}
			
			and
			\begin{equation}\label{Eq:CJ_k}
				\mathcal{CJ}_k = \sum_{\substack{i\geq1 \\
						\alpha \in \Integer/m\Integer}} \left(i p_i^{(\alpha+k)}
				\frac{\partial}{\partial
					p^{(\alpha)}_i} \right) 
				\qquad \text{for } k = 1,\dots,m-1
			\end{equation}
			
		\end{theorem}
		\begin{remark}
			The operator $\mathcal{CJ}_0$ has $2m^2$ terms, while 
			$\mathcal{CJ}_k$ has $m$ terms for each $k$.
		\end{remark}
		
		\begin{example}
			For the wreath product $\Integer/2\Integer \wr S_n$, $\alpha$ and $\gamma$ can be
			equal only to $0$ or $1$. Let $p^{(0)}_i\bydef p_i$ and $p^{(1)}_i\bydef q_i$.
			Thus for $\mathcal{CJ}_0$ one gets:
			\begin{align*}
				\mathcal{CJ}_0 =&\frac{1}{2} \sum_{\substack{i,j \ge 1}}
				(i+j)\left( p_i p_j \pder{}{p_{i+j}} +p_i q_j \pder{}{q_{i+j}}+q_i p_j \pder{}{q_{i+j}}+
				q_i q_j \pder{}{p_{i+j}}\right)\\
				&+2ij\left( p_{i+j} \frac{\partial^2}{\partial p_i p_j}+
				p_{i+j} \frac{\partial^2}{\partial q_i q_j}+q_{i+j} \frac{\partial^2}{\partial q_i p_j}+
				q_{i+j} \frac{\partial^2}{\partial p_i q_j}\right)
			\end{align*} 
			which after simplification gives 
			\begin{multline}
				\CJ_0 = \sum_{i,j=1}^\infty \biggl( (i+j) p_iq_j \pder{}{q_{i+j}} +
				2ij q_{i+j} \pdertwo{}{p_i}{q_j} + ijp_{i+j} \pdertwo{}{q_i}{q_j}\\
				+ \frac{1}{2} (i+j) q_iq_j \pder{}{p_{i+j}} + \frac{1}{2} (i+j)
				p_ip_j \pder{}{p_{i+j}} + ij p_{i+j} \pdertwo{}{p_i}{p_j}\biggr)
			\end{multline}
			and we recover the result from \cite{B-Hurwitz}.
			
			Similarly for $\mathcal{CJ}_1$ we get 
			\begin{equation}\label{Eq:CJ_2}
				\CJ_1 = \sum_{i=1}^\infty \left(i p_i\frac{\partial}{\partial
					q_i} + i q_i\frac{\partial}{\partial p_i}\right)
			\end{equation}
			which also agrees with \cite{B-Hurwitz}
		\end{example}
		
		To prove theorem \ref{Th:CJ} we will explicitly compute the coefficients 
		$\langle \overline{\lambda}\mid \overline{\mu}\rangle_\nu$
		for all possible $\overline{\lambda}$, $\overline{\mu}$ and 
		$\nu=0,\dots,m-1$.
		
		Let $\sigma \in S_n$, and $1\leq a<b\leq n$. The cyclic structure 
		of the product $\sigma(a,b)$ depends on the position of $a$ and $b$: 
		if $a$ and $b$ are in the same cycle, say $\sigma_1$, then $\sigma_1$ is
		split, we call this a cut.
		If $a$ and $b$ are in different cycles say $\sigma_1$ and $\sigma_2$
		respectively, then these two cycles are joined.
		
		Let now $\sigma\in \G$ and $\rho$ being one of the following
		types of reflections (where elements are modulo $(m-1)n$):
		\begin{itemize}
			\item  $\rho=r_{a,b}=(a,b)(\tau(a),\tau(b))\dots
			(\tau^{m-1}(a),\tau^{m-1}(b))$ for $1\leq a,b\leq (m-1)n$
			
			\item $\rho = \ell_a^k = (a,\tau(a),\tau^{2}(a),\dots)^k$
			for $1\leq a \leq (m-1)n$ and for any $k =1,\dots, m-1$
		\end{itemize}
		The structure of $\sigma'=\rho\sigma $ depends on the structure of $\sigma$, i.e
		which $\beta_\alpha$-cycle $\sigma$ is, and the positions of $a$ and $b$.
		
		\subsubsection{The operators $\mathcal{CJ}_0$}
		
		Let $\alpha$ and $\gamma$ such that  $0\leq \alpha,\gamma \leq m-1$ 
		and let the permutation $\sigma$ with cycle decomposition
		the product of a $\beta_\alpha-$cycle and a $\beta_\gamma$-cycle 
		where the $\beta_\alpha$-cycle is:
		
		$(a_1,\dots a_{\ell_1},\tau^\alpha(a_1),\dots,\tau^\alpha(a_{\ell_1}),\dots)(\tau(a_1),
		\dots,\tau(a_{\ell_1}),\tau^{\alpha+1}(a_1),\dots,\tau^{\alpha+1}(a_{\ell_1}),\dots)$ 
		
		and the $\beta_\gamma$-cycle is:
		
		$(b_1,\dots b_{\ell_2},\tau^\gamma(b_1),\dots,\tau^\gamma(b_{\ell_2}),\dots)(\tau(b_1),
		\dots,\tau(b_{\ell_2}),\tau^{\gamma+1}(b_1),\dots,\tau^{\gamma+1}(b_{\ell_2}),\dots)$ 
		
		Let also the reflection 
		$\rho_{a_1,b_1} = (a_1,b_1)(\tau(a_1),\tau(b_1))\dots
		(\tau^{m-1}(a_1),\tau^{m-1}(b_1))$ where $1\leq a_1,b_1\leq (m-1)n$.
		
		A straightforward computation shows that%:
		
		%	  $\rho_{a_1,b_1}\sigma = (a_1,\dots,a_{\ell_1-1,} \tau^\alpha(b_1),\dots,
		%	 \tau^{\gamma+\alpha}(a_1),\dots)(\tau(a_1),\dots,\tau(a_{\ell_1-1}), \tau^{\alpha+1}(b_1),\dots,
		%	 \tau^{\gamma+\alpha+1}(a_1),\dots)$
		%	 
		%	 i.e 
		$\rho\sigma$ is a $\beta_{\alpha+\gamma}-$cycle.
		Indeed
		
		\begin{equation}
			\begin{aligned}
				&\sigma\rho(a_{\ell_1})=\tau^\alpha(b_1)\\
				&\sigma\rho(\tau^\alpha(b_1))=\tau^\alpha(b_2)\\
				&\vdots\\
				&\sigma\rho(\tau^\alpha(b_{\ell_2}))=
				(\tau^{\gamma+\alpha}(a_1))\\
			\end{aligned}
		\end{equation}
		
		In other words, multiplying by $\rho$ a product of a $\beta_\alpha$ with a
		$\beta_\gamma$-cycles
		gives a $\beta_{\alpha+\gamma}$-cycle, we call this a \textit{join}. 
		Furthermore $\rho$ being an involution,
		multiplying the $\beta_{\alpha+\gamma}$-cycle obtained 
		once more by $\rho$, gives back the product of $\beta_\alpha\beta_\gamma$-cycles,
		we call this a \textit{cut}. 
		There are $m^2$ possible joins and $m^2$ possible cuts counted
		without multiplicities, so $2m^2$ cases in total.
		This being true for any $\alpha$, $\beta$ and $\alpha+\beta$, these are
		all the possible cases for the multiplication by a reflection 
		$\rho = (a_1,b_1)(\tau(a_1),\tau(b_1))\dots
		(\tau^{m-1}(a_1),\tau^{m-1}(b_1))$. 
		
		We now compute the multiplicities 
		$\langle \overline{\lambda}\mid \overline{\mu}\rangle_0$,
		which will prove the result for $\mathcal{CJ}_0$.
		Let $\sigma\in C_{\overline{\lambda}}$, where
		$\overline{\lambda}=(\lambda_0,\dots,\lambda_{m-1})$ and
		$\lambda_\alpha= 1^{c_{1\alpha}}\dots n^{c_{n\alpha}}$ that is for every
		$k=1,\dots n$ the element $\sigma$ contains $c_k$ $\beta_\alpha$-cycles
		of length $k\times \frac{m}{\gcd(m,\alpha)}$.
		In the same way define $\sigma\rho \in C_{\overline{\mu}}$.
		Calculate $\langle \overline{\lambda}\mid \overline{\mu}\rangle_0$
		for all $\overline{\mu}=(\mu_0,\dots,\mu_{m-1})$ where
		$\mu_\gamma = 1^{d_{1\gamma}}\dots n^{d_{n\gamma}}$.

		{\def \labelenumi {Case \theenumi}
			
			\begin{enumerate}  
				
				\item\label{case:Cut} Cut: Let a $\beta_{\alpha+\gamma}$-cycle of length 
				$(i+j)\times \frac{m}{\gcd(m,\alpha+\gamma)}$.
				The total possible numbers of
				positions for $a_1$ is the total number of elements in all
				the $\beta_{\alpha+\gamma}$-cycle of length 
				$(i+j)\times \frac{m}{\gcd(m,\alpha+\gamma)}$, that is
				$(i+j)\times \gcd(m,\alpha+\gamma)\times
				\frac{m}{\gcd(m,\alpha+\gamma)} c_{i+j}=
				m(i+j)c_{i+j}$. 
				Since we know the length of the product of 
				$\beta_\alpha$ with the $\beta_\gamma$ of $\rho\sigma$,
				the position of $b_1$ as well as the positions of 
				$\tau^\nu(a_1)$ and $\tau^\nu(b_1)$ for $\nu =1,\dots m-1$
				is unique once the position of $a_1$ is chosen.
				Doing like this, one counts every reflection $
				\rho_{a_1,b_1}$ $m$ times, we further
				divide by two as $\rho_{a_1,b_1}=\rho_{b_1,a_1}$,
				and the multiplicity is $ \frac{1}{2}(i+j)c_{i+j}$
				\medskip
				
				\item \label{case:Join} Join: Let a $\beta_{\alpha}$-cycle of length 
				$i\times \frac{m}{\gcd(m,\alpha)}$ and a $\beta_{\gamma}$-cycle of length 
				$j\times \frac{m}{\gcd(m,\gamma)}$.
				The total possible numbers of
				positions for $a_1$ is the total number of elements in all
				the $\beta_{\alpha}$-cycle of length 
				$i\times \frac{m}{\gcd(m,\alpha)}$, that is
				$i\times \gcd(m,\alpha)\times
				\frac{m}{\gcd(m,\alpha)} c_{i}=
				mic_{i}$. Similarly for $b_1$ we get $mjc_{j}$
				total possible number of positions.
				For the same reason as above we need to divide by 
				$2m$ to not over count $\rho_{a_1,b_1}$. We get that the multiplicity is 
				$mijc_ic_j$
				
			\end{enumerate}
		}
		
		The explicit formula of $\mathcal{CJ}_0$ follows from Theorem \ref{Th:TViaMult}
		as \begin{align*}
			\mathcal{CJ}_0p_{\overline{\lambda}}=
			\sum_{\overline{\mu}} \langle \overline{\lambda}
			\mid \overline{\mu}\rangle_0p_{\overline{\mu}}
		\end{align*}
		
		For the cut: the monomial $p_{\overline{\lambda}}$ contains $p_{i+j}^{(\alpha+\beta)c_{i+j}}$.
		The exponent of $p_{i+j}^{(\alpha+\beta)c_{i+j}}$ in the monomial $p_{\overline{\mu}}$ is decreased
		by one, and the exponent of $p_{i}^{(\alpha)}$ and $p_{j}^{(\beta)}$ are increased by one.
		We will have $p^{(\alpha)}_i p^{(\gamma)}_j \pder{}{p^{(\alpha+\gamma)}_{i+j}}$
		for the cut, and to have the correct multiplicity add the coefficient
		$\frac{i+j}{2}$ before. 
		The join case is similar, and this finishes to proof for $\CJ_0$.

		\subsubsection{The operators $\mathcal{CJ}_k$}\label{MultByl}
		
		The reasoning to get the explicit formulas for $\mathcal{CJ}_k$ is 
		similar to the one for $\CJ_0$, but looking instead at the
		product $\sigma\rho$ where $\rho=(a_1,\tau(a_1),\dots \tau^{m-1}(a_1))^k$
		for $k=1,\dots,m-1$.
		\medskip
		
		Once more let $\sigma$ be the $\beta_\alpha$-cycle:
		
		\noindent$(a_1,\dots a_{\ell_1},\tau^\alpha(a_1),\dots,
		\tau^\alpha(a_{\ell_1}),\dots)(\tau(a_1),
		\dots,\tau(a_{\ell_1}),\tau^{\alpha+1}(a_1),\dots,
		\tau^{\alpha+1}(a_{\ell_1}),\dots)$.
		Let also the reflection $\rho$:
		
		$\rho = \ell_{a_1}^\nu = (a_1,\tau(a_1)\dots,\tau^{m-1}(a_1))^\nu$
		for some $\nu=1,\dots,m-1$.
		
		A straightforward computation shows that  
		\noindent$\sigma\rho$ % =(a_1,\dots a_{\ell_1},\tau^{\alpha+\nu}(a_1),\dots,
		%	\tau^{\alpha+\nu}(a_{\ell_1}),\dots)(\tau(a_1),
		%	\dots,\tau(a_{\ell_1}),\tau^{\alpha+\nu+1}(a_1),\dots,
		%	\tau^{\alpha+\nu+1}(a_{\ell_1}))\dots$		
		%	\noindent In other words, multiplying a $\beta_\alpha$-cycle by a reflection
		%	$\ell_{a_1}^\gamma$ gives 
		is a $\beta_{\alpha+\gamma}-$cycle.
		
		We now compute the multiplicities 
		$\langle \overline{\lambda}\mid \overline{\mu}\rangle_\gamma$,
		for $\gamma = 1,\dots, m-1$
		which will end the proof of theorem \ref{Th:CJ}.
		As above, let $\sigma\in C_{\overline{\lambda}}$, where
		$\overline{\lambda}=(\lambda_0,\dots,\lambda_{m-1})$ and
		$\lambda_\alpha= 1^{c_{1\alpha}}\dots n^{c_{n\alpha}}$ that is for every
		$k=1,\dots, n$ the element $\sigma$ contains $c_k$ $\beta_\alpha$-cycles
		of length $k\times \frac{m}{\gcd(m,\alpha)}$.
		In the same way define $\sigma\rho \in C_{\overline{\mu}}$.
		Calculate $\langle \overline{\lambda}\mid \overline{\mu}\rangle_\nu$
		for all $\overline{\mu}=(\mu_0,\dots,\mu_{m-1})$ where
		$\mu_\gamma = 1^{d_{1\gamma}}\dots n^{d_{n\gamma}}$ and $\nu= 1,\dots,m-1$.
		
		Let a $\beta_{\alpha}$-cycle of length 
		$i\times \frac{m}{\gcd(m,\alpha)}$.
		The total possible numbers of
		positions for $a_1$ is the total number of elements in all
		the $\beta_{\alpha}$-cycle of length 
		$i\times \frac{m}{\gcd(m,\alpha)}$, that is
		$i\times \gcd(m,\alpha)\times
		\frac{m}{\gcd(m,\alpha)} c_{i}=
		mic_{i}$. We need to divide the multiplicity by $m$
		to not over count the reflection
		$\ell^\nu_{a_1}$ as
		$\ell^\nu_{a_1}=\ell^\nu_{\tau(a_1)}=\dots=
		\ell^\nu_{\tau^{m-1}(a_1)}$, and finally the multiplicity
		is $ic_i$
		
		As done for $\CJ_0$, the monomial $p_{\overline{\lambda}}$ 
		contains $p_{i}^{(\alpha)c_{i}}$.
		The exponent of $p_{i+j}^{(\alpha)c_{i}}$ in the monomial
		$p_{\overline{\mu}}$ is decreased by one, and the exponent 
		of $p_{i}^{(\alpha+k)}$ is increased by one.
		We finally get $p^{(\alpha+k)}_i\pder{}{p^{(\alpha)}_{i}}$,
		and to have the correct multiplicity add the coefficient
		$i$ before and theorem \ref{Th:CJ} is proved. 
		
		\subsection{Change of variables} We show in this section,
		that after a suitable change of variable, the operators in (\ref{Eq:CJ_0})
		and in (\ref{Eq:CJ_k}) reduce to the classical cut-and-join
		
		denote by \begin{equation}\label{Eq:CJA}
			\mathcal{CJ} \bydef \frac{1}{2}\sum_{i,j=1}^\infty \biggl( ijp_{i+j}\pdertwo{}{p_i}{p_j} +
			(i+j) p_ip_j \pder{}{p_{i+j}}\biggr)
		\end{equation}
		and Euler fields
		\begin{equation}\label{Eq:Euler}
			E\bydef\sum_{i=1}^\infty ip_i \pder{}{p_i},
		\end{equation}
		respectively.
		
		\begin{proposition}\label{Prop:CJinU}
			Let $\xi^k\bydef \exp(\frac{2ik\pi}{m})$. Let now the 
			change of variable 
			\begin{align*}
				&p^{(0)}_i= u^{(0)}_i+u^{(1)}_i+\dots+u^{(m-1)}_i\\
				&p^{(1)}_i= u^{(0)}_i+\xi u^{(1)}_i+\dots+\xi^{m-1}u^{(m-1)}_i\\
				&\vdots\\
				&p^{(m-1)}_i= u^{(0)}_i+\xi^{m-1} u^{(1)}_i+\dots+\xi u^{(m-1)}_i.
			\end{align*}
			Then 
			\begin{align*}
				\mathcal{CJ}_0 &= m(\CJ_{u^{(0)}} +\dots+ \CJ_{u^{(m-1)}}),\\
				\mathcal{CJ}_k &=  \sum_{\alpha=0}^{m-1}	\xi^{k\alpha}E_{u^{(\alpha)}}
				\bydef \mathcal{E}_{k}
				\qquad k = 1,\dots,m-1
			\end{align*}
			
			\noindent where by $\CJ_{u^{(\alpha)}}$ we denote the operator \eqref{Eq:CJA}
			with $u_i^{(\alpha)}$ substituted for $p_i$, and similarly, $E_{u^{(\alpha)}}$.
		\end{proposition}
		
		\begin{proof}
			First for $\mathcal{CJ}_0$. By the chain rule
			\begin{equation}\label{Eq:ChangeVar}
				\frac{\partial}{\partial p^{(\alpha)}_i} =
				\frac{1}{m}\sum_{\nu=0}^{m-1}\xi^{-\nu\alpha}\frac{\partial}{\partial u^{(\nu)}_i}.
			\end{equation}
			Then plugging (\ref{Eq:ChangeVar}) into (\ref{Eq:CJ}) gives
			\begin{align}\label{Eq:CJ_0NewVar}
				\mathcal{CJ}_0 =\frac{1}{2m} \sum_{\substack{i,j \ge 1\\ 
						\alpha,\gamma\in \Integer/m\Integer\\
						\nu_1,\nu_2,\nu_3 \in \Integer/m\Integer
				}}
				\xi^{\alpha\nu_1+\beta\nu_2-(\alpha+\beta)\nu_3}
				(i+j) u^{(\nu_1)}_i u^{(\nu_2)}_j \pder{}{u^{(\nu_3)}_{i+j}} \\ 
				+\xi^{(\alpha+\beta)\nu_1 -\alpha\nu_2-\beta\nu_3}mij u_{i+j}^{(\nu_1)} 
				\frac{\partial^2}{\partial u^{(\nu_2)}_iu^{(\nu_3)}_j}
			\end{align}
			We want to express the coefficients in front of the monomials  
			$(i+j) u^{(\nu_1)}_i u^{(\nu_2)}_j \pder{}{u^{(\nu_3)}_{i+j}}$ and 
			$mij u_{i+j}^{(\nu_1)} 
			\frac{\partial^2}{\partial u^{(\nu_2)}_iu^{(\nu_3)}_j}$.
			A straightforward computation gives:

			\begin{equation*}
				\sum_{\alpha,\gamma\in \Integer/m\Integer\\}
				\xi^{\alpha\nu_1+\beta\nu_2-(\alpha+\beta)\nu_3} = 
				\sum_{\alpha,\gamma\in \Integer/m\Integer\\}
				\xi^{\alpha(\nu_1-\nu_2)+\beta(\nu_1-\nu_3)} =
				\begin{cases}
					m^2 \ \text{if} \ \nu_1=\nu_2=\nu_3\\ 
					0\ \text{otherwise.}
				\end{cases}
			\end{equation*}
			
			Plugging this last equality into (\ref{Eq:CJ_0NewVar}) gives
			
			\begin{align*}
				\mathcal{CJ}_0 =m \sum_{\substack{i,j \ge 1\\ 
						\nu\in \Integer/m\Integer
				}}
				(i+j) u^{(\nu)}_i u^{(\nu)}_j \pder{}{u^{(\nu)}_{i+j}}
				+mij u_{i+j}^{(\nu)} 
				\frac{\partial^2}{\partial u^{(\nu)}_iu^{(\nu)}_j}
			\end{align*}
			and this finishes the proof for $\mathcal{CJ}_0$
			
			By the same reasoning, for $\mathcal{CJ}_k$ one gets:
			\begin{equation}\label{Eq:CJ_kNewVar}
				\mathcal{CJ}_k = \frac{1}{m}\sum_{\substack{i\geq1 \\
						\alpha,\nu_1,\nu_2 \in \Integer/m\Integer}} 
				\xi^{(\alpha+k)\nu_1-\alpha\nu_2}u^{(\nu_1)}_i
				\frac{\partial}{\partial
					u^{(\nu_2)}_i}
			\end{equation}
			this reduces to compute 
			
			\begin{equation*}
				\sum_{\substack{\alpha\in \Integer/m\Integer}} 
				\xi^{(\alpha+k)\nu_1-\alpha\nu_2} = 
				\xi^{k\nu_1}\sum_{\alpha\in \Integer/m\Integer}
				\xi^{\alpha\nu_1-\alpha\nu_2} =
				\begin{cases}
					m \ \text{if} \ \nu_1=\nu_2\\ 
					0\ \text{otherwise}
				\end{cases}
			\end{equation*}
			and plugging it into (\ref{Eq:CJ_kNewVar}) finishes the proof.
		\end{proof}
		
		\begin{theorem}\label{Th:G(m,1,n)-Schur}
			
			For all colored partitions $\overline{\lambda}=\lambda_0|\dots|\lambda_{m-1}$
			the polynomials
			\begin{equation*}
				s_{\overline{\lambda}}(\mathbf{P}^{(0)},\dots,\mathbf{P}^{(m-1)}) \bydef
				s_{\lambda_0}(\mathbf{P}^{(0)})\dots s_{\lambda_{m-1}}(\mathbf{P}^{(m-1)})
			\end{equation*}  
			where $s_{\lambda_i}$ are Schur polynomials and $\mathbf{P}^{(\alpha)}$ 
			means $(\frac{\sum_{\beta} \xi^{\alpha \beta}p_{1}^{(\alpha)}}{m},
			\frac{\sum_{\beta} \xi^{\alpha \beta}p_{2}^{(\alpha)}}{m}, \dots)$, are
			eigenvectors of  $\CJ_0$ and $\CJ_k$ for all $k=1,\dots, m-1$;
			the respective eigenvalues
			are $c_0(\overline{\lambda})\bydef m\sum\limits_{i=1}^\infty \left(\lambda_{0,i}(\lambda_{0,i}-2i+1) +
			\lambda_{1,i}(\lambda_{1,i}-2i+1) + \dots +
			\lambda_{m-1,i}(\lambda_{m-1,i}-2i+1)\right)$ for $\CJ_0$ and 
			$c_k(\overline{\lambda})\bydef\sum\limits_{i=1}^\infty
			(\lambda_{0,i} + \xi^{k}\lambda_{1,i}+\dots+\xi^{k(m-1)}\lambda_{m-1,i})$ for $\mathcal{CJ}_k$.
		\end{theorem}
		
		The theorem follows immediately from Proposition \ref{Prop:CJinU} and
		the fact that Schur polynomials are eigenvectors of the
		cut-and-join operator \eqref{Eq:CJA} (see \cite{LandoKazarian}) and of
		the Euler field \eqref{Eq:Euler} ($s_{\overline{\lambda}}$ are weighted
		homogeneous).
		
		Theorem \ref{Th:G(m,1,n)-Schur} and the Cauchy identity \cite{Macdonald}
		imply that
		\begin{equation*}
			e^{p_1^{(0)}} = e^{u^{(0)}_1}\dots e^{u^{(m-1)}_1} = \sum_{\overline{\lambda}}
			s_{\overline{\lambda}}(\mathbf{P}^{(0)},\dots,\mathbf{P}^{(m-1)})
			s_{\overline{\lambda}}(1, 0, \dots; 1, 0, \dots)
		\end{equation*}
		
		This allows us to prove the following
		
		\begin{corollary}[of Theorem \ref{Th:G(m,1,n)-Schur}]\label{Cr:LKGenFunc}
			\begin{align*}
				&\mathcal{H}(\beta_0,\beta_1,\dots,\beta_{m-1},\mathbf{p})= \\
				& \sum_{\overline{\lambda}}
				\exp\bigl(\beta_0 c_0(\overline{\lambda})+ \sum_{1\leq k\leq m-1} \beta_k c_k(\overline{\lambda}))
				\times s_{\overline{\lambda}}(1,0,\dots; 1,0,\dots)
				s_{\overline{\lambda}}(\mathbf{P}^{(0)},\dots,\mathbf{P}^{(m-1)})
			\end{align*}
		\end{corollary}
		This is the $G(m,1,n)$-analog of the formula expressing the simple Hurwitz
		numbers via the Schur polynomials \cite{LandoKazarian}

		\section{KP Hierarchy}
		
		The KP hierarchy is one of the best-known and studied integrable
		systems; see e.g.\cite{Kramer,KazarianHodge,FSpace} for details. It
		is an infinite system of PDE applied to a formal series $F \in
		\Complex[[t]]$ where $t = (t_1, t_2, \dots)$ is a countable collection
		of variables (``times''). Each equation looks like
		$\pdertwo{F}{t_i}{t_j} = P_{ij}(F)$ where $P_{ij}$ is a homogeneous
		multivariate polynomial with partial derivatives of $F$ used as its
		arguments. For example, the first two equations of the hierarchy are
		
		\begin{align*}
			&F_{22}=-\frac{1}{2}F_{11}^2 + F_{31} -\frac{1}{12}F_{1111}\\
			&F_{32}=-F_{11}F_{21}+F_{41}-\frac{1}{6}F_{2111};
		\end{align*}
		here $F_{i_1i_2\dots i_n}$ means $\frac{\partial^nF}{\partial
			t_{i_1}\partial t_{i_2\dots\partial t_{i_n}}}$. If $F$ is a solution
		of the hierarchy, its exponential $\tau= e^F$ is called a
		$\tau$-function.
		
		\begin{example}\label{Ex:T1Tau}
			Equations of the hierarchy do not contain first-order derivatives,
			so the function $F(t) = t_1$ is a solution of the KP hierarchy;
			$\tau = e^{t_1}$ is a $\tau$-function.
		\end{example}
		
		\begin{proposition}\label{Pp:Sato}
			There exists a Lie algebra $\mathcal G$ of differential operators on
			the space $\Complex[[p]]$ having the following properties:
			\begin{enumerate}
				\item\label{It:GLInf} $\mathcal G$ is isomorphic to a
				one-dimensional central extension $\widehat{\name{gl}(\infty)}$ of
				a Lie algebra $\name{gl}(\infty)$ of infinite matrices
				$(x_{ij})_{i,j \in \Integer}$ having nonzero elements at a finite
				set of diagonals.
				
				\item\label{It:F2F3} The Euler field $E$ and the cut-and-join
				operator $\CJ$ belong to $\mathcal G$.
				
				\item\label{It:Orbit} For any $\tau$-function $f$ and any $X \in
				\mathcal G$, the power series $e^X f$ is a $\tau$-function,
				too. 
			\end{enumerate}
		\end{proposition}
		
		See \cite{Kramer} for the definition of $\mathcal G$, \cite{2Sato} and
		\protect{[section 6]\cite{KazarianHodge} for the proof.
			
			\begin{theorem}\label{Th:HGsolKP}
				The generating function 
				$\mathcal{H}(\beta_0,\beta_1,\dots,\beta_{m-1},\mathbf{p})$ is a
				$m$-parameter family of $\tau$-functions, independently in the $p^{(i)}$
				variables.
			\end{theorem} 
			
			\begin{proof}
				Denote by $\mathbf{u}^{\alpha}_i\bydef\sum_{0\leq\beta\leq m-1}
				\xi^{\alpha\beta}u_i^\beta$.
				Equation \eqref{Eq:GenFunc} and Proposition \ref{Prop:CJinU} imply that
				\begin{align*}
					&\mathcal{H}(\beta_0,\dots,\beta_{m-1},\mathbf{u}^{0},\dots,\mathbf{u}^{m-1}) \\
					&= e^{\beta_0 m(\CJ_{u^{(0)}}+\dots+
						\CJ_{u^{(m-1)}})}e^{(\beta_1\mathcal{E}_1+\dots
						\beta_{m-1}\mathcal{E}_{m-1})} e^{u^{(0)}_1}\dots e^{u^{(m-1)}_1} \\
					&=e^{(\beta_0\CJ_{u^{(0)}}+\beta_1E_{u^{(0)}}+\dots+\beta_{m-1}E_{u^{(0)}})}e^{u^{(0)}}\dots
					e^{\beta_0\CJ_{u^{(m-1)}}+\beta_1\xi^{m-1}E_{u^{(m-1)}}+\dots+\beta_{m-1}\xi E_{u^{(m-1)}}}e^{u^{(m-1)}}
				\end{align*}
				
				(operators $\CJ_{u^{(\alpha)}}$, $\CJ_{u^{(\nu)}}$,$E_{u^{(\alpha')}}$ 
				and $E_{u^{(\nu')}}$ commute because they act on different
				sets of variables).
				
				By assertion \ref{It:F2F3} of Proposition \ref{Pp:Sato},
				$\beta_0\CJ_{u^{(0)}}+\beta_1E_{u^{(0)}}+\dots+\beta_{m-1}E_{u^{(0)}}
				\in \mathcal G_u$ (the Lie algebra $\mathcal
				G$ acting by differential operators on the $u^{(0)}$ variables). By Example
				\ref{Ex:T1Tau}, $e^{u^{(0)}_1}$ is a $\tau$-function 
				(in the $\mathbf{u}^{(0)}$ variables),
				and so is $\mathcal{H}^B(\beta_0,\dots,\beta_{m-1},\mathbf{u}^{0},\dots,\mathbf{u}^{m-1})$
				by assertion
				\ref{It:Orbit} of the same proposition. The reasoning for the $u^{(\alpha)}$
				with $\alpha= 1,\dots,m-1$ variables is the same.
			\end{proof}

			\section{Hurwitz numbers and ramified covering}
			\subsection{Wreath ramified coverings}
			In \cite{Zhang} the notion of wreath branched cover for a general wreath
			product of $S_n$ is introduced, we will follow it closely.
			For details and proofs see \cite{Zhang}.
			
			In section 1.2 we showed the correspondence between the conjugacy classes
			of $G(m,1,n)$ when embedded in $S_{mn}$, and the conjugacy classes
			of $\Integer/m\Integer  \wr S_n$. In this section  conjugacy classes are
			subsets of the wreath product $\Integer/m \Integer  \wr S_n$.
			
			\noindent A $G(m,1,n)-$ramified cover $(p,f)$ of degree $mn$ over $\Complex P^1$
			an ordinary ramified cover  $f$ of degree $n$: 
			\begin{align*}
				f : M \to \Complex P^1
			\end{align*} 
			with a set of critical values $Y$ together with a 
			smooth curve $\tilde{M}$ with a left $\Integer/m\Integer-$action and 
			$\Integer/m\Integer$-equivariant map ($\Integer /m \mathbb Z$ acts trivially on $M$)
			$
			p: \tilde{M}\to M
			$
			such that 
			\begin{equation*}
				p' : \tilde{M} - p^{-1}(f^{-1}(Y))\to M - f^{-1}(Y),
			\end{equation*}
			where $p'$ is the restriction of $p$
			to $\tilde{M} - p^{-1}(f^{-1}(Y))$, is a principal  $\mathbb Z/m\mathbb Z$- bundle map.
			
			Two $G(m,1,n)-$ramified covers
			\begin{equation*}
				F: \tilde{M} \xrightarrow{{p}} M
				\xrightarrow{{f}}\Complex P^1
			\end{equation*}
			and 
			\begin{equation*}
				F': \tilde{M'} \xrightarrow{{p'}} M'
				\xrightarrow{{f'}} \Complex P^1
			\end{equation*}
			
			\noindent are said to be equivalent if there is a pair of biholomorphic maps
			$\tilde{\phi}: \tilde{M}\to\tilde{M}' $ and $\phi: M\to M'$
			such that $\tilde{\phi}$ is $\Integer/m\Integer$-equivariant and the
			following diagram commutes
			\begin{center}
				\begin{tikzcd}
					\tilde{M} \arrow[d, "\tilde{\phi}"] \arrow[r,"p"]     & M \arrow[r,"f"] \arrow[d, "\phi"]    & \Complex P^1 \arrow[d, "id"]\\
					\tilde{M}' \arrow[r,"p'"]                         & M' \arrow[r,"f'"]                       & \Complex P^1.
				\end{tikzcd}
			\end{center}
			
			An automorphism of a $G(m,1,n)-$ramified cover $(f,p)$ is a pair of biholomorphic maps
			$\tilde{\phi}: \tilde{M}\to\tilde{M} $ and $\phi: M\to M$ such that $\tilde{\phi}$
			is $\Integer/m\Integer$-equivariant and the 
			following diagram commutes
			\begin{center}
				\begin{tikzcd}
					\tilde{M} \arrow[d, "\tilde{\phi}"] \arrow[r,"p"]     & M \arrow[r,"f"] \arrow[d, "\phi"]    & \Complex P^1 \arrow[d, "id"]\\
					\tilde{M} \arrow[r,"p"]                         & M \arrow[r,"f"]                       & \Complex P^1.
				\end{tikzcd}
			\end{center}
			
			An important result of \cite{Zhang} is the description of the monodromy 
			of $G(m,1,n)-$ra\-mi\-fied covers. We describe it now. Let $(f,p)$ a
			$G(m,1,n)-$ramified cover
			\begin{equation*}
				\tilde{M} \xrightarrow{{p}} M
				\xrightarrow{{f}}\Complex P^1.
			\end{equation*}
			
			Taking a critical value $y\in Y$, the preimage $f^{-1}(y)$ is a set
			of points $x_1,\dots,x_\ell$ with multiplicities $\lambda_1,\dots,\lambda_\ell$,
			where $\lambda_1+\lambda_2+\dots+\lambda_\ell=n$, i.e a partition of $n$. 
			We can assume without loss  of generality that $\lambda_1\leq\lambda_2\leq
			\dots\leq \lambda_\ell$ . Furthermore pick a regular point $b$ in $\Complex P^1$,
			its preimage by $f$ are $n$ points $q_1,\dots,q_n$.
			The map $p'$ being a normal cover, we can identify each fiber 
			$p^{-1}(q_i)$ with $\Integer/m\Integer$ using local trivialization:
			choose a small neighborhood $U_i$ of $q_i$ and a diffeomorphism 
			
			\begin{align*}
				\psi_i: p^{-1}(U_i)\to U_i\times\Integer/m\Integer, \ z\to (p(z),\phi_i(z))
			\end{align*}
			such that $\phi_i(hz)=h\phi_i(z)$ for any $h \in\Integer/m\Integer $
			
			Let a loop $\gamma$ based on $b$ and winding once around $y$.
			The monodromy of $f$ along the loop
			$\gamma$ is a permutation $(q_{1,1},\dots,q_{1,\lambda_1})\dots(q_{\ell,1},\dots,
			q_{\ell,\lambda_\ell})$ of $S_n$ belonging to the conjugacy class indexed by $\lambda$.
			Pick now one of the cycles, say $(q_{j,1},\dots,q_{j,\lambda_j})$,
			and denote by $\gamma_i$ the lifting of $\gamma$ starting at $q_{j,i}$.
			Now starting with an element $z_1\in p^{-1}(q_{j,1})$, the unique lifting 
			$\tilde{\gamma_1}$ of $\gamma_1$ will end at some $z_2\in p^{-1}(q_{j,2})$,
			continuing so we get that the lifting $\tilde{\gamma_{\lambda_j}}$ of $\gamma_{\lambda_j}$
			will end at some $z_1'\in p^{-1}(q_{j,1})$. Suppose 
			
			\begin{align*}
				\phi_2(z_2)=g_1\phi_1(z_1), \dots, \phi_{\lambda_j}(z_{\lambda_j})=
				g_{\lambda_{j-1}}\phi_{\lambda_{j-1}}(z_{\lambda_{j-1}}), \
				\phi_1(z_1')=g_{\lambda_j}\phi_{\lambda_j}(z_{\lambda_j}))
			\end{align*}
			for $g_1,\dots,g_{\lambda_j}\in \Integer/m\Integer,$ then
			$$g_{\lambda_j}g_{\lambda_{j-1}}\dots g_1(z_1)=\phi_1(z_1')$$.
			
			The product $g_{\lambda_j}g_{\lambda_{j-1}}\dots g_1(z_1)=\phi_1(z_1')\in \Integer/m\Integer$
			can be viewed as the monodromy of $p$ induced by the loop $\gamma_1\dots\gamma_{\lambda_j}$
			starting and ending at $q_{j,1}$.
			
			Add the following restriction that $\phi_1(z_1)=0\in \Integer/m\Integer$.
			
			\begin{lemma}\cite{Zhang}
				The cycle product $g_{\lambda_j}g_{\lambda_{j-1}}\dots g_1\in \Integer/m\Integer$
				is determined by $\gamma_1\dots\gamma_{\lambda_j}$ up to conjugacy.
			\end{lemma}
			
			Now given a loop $\gamma$ winding once around a critical value $y$, we get an element
			$\sigma\in S_n$, together with an element $g=(g_1,\dots,g_n)\in (\Integer/m\Integer)^n$.
			In other word we get a map 
			\begin{align*}
				\mathfrak{M} :\ \pi_1(\Complex P^1 - Y)&\to \Integer/\Integer m \wr S_n\\
				&[\gamma] \to (\sigma,g)
			\end{align*}
			which is a group homomorphism. This is the monodromy representation associated to 
			with the $G(m,1,n)-$cover $(p,f)$
			
			\begin{lemma}\cite{Zhang}
				The monodromy group of a $G(m,1,n)-$ramified cover $(p,f)$ of 
				degree $mn$ over $\Complex P^1$ is a subgroup of the wreath product
				$\Integer/\Integer m \wr S_n$, which is determined up to conjugacy class. And it is transitive if $M$ is connected.
			\end{lemma}
			
			The inverse construction is as follows:
			given a group homomorphism \\
			$\mathfrak{M} : \pi_1(\Complex P^1 - Y)\to \Integer/\Integer m \wr S_n$,
			we first construct a ramified cover $f : M\to \Complex P^1$ of degree $n$,
			using the composition homomorphism of $\mathfrak{M}$ and the projection map
			$\Integer/\Integer m \wr S_n \to S_n$ as monodromy representation. Then
			we construct a cover $p : \tilde{M} \to M$ using the information resulting
			from the cycle products. This produce a $G(m,1,n)-$ramified cover $(p,f)$ of 
			degree $mn$, where $\mathfrak{M}$ is its monodromy representation.
			
			The construction described above gives:
			
			\begin{lemma}\cite{Zhang}
				The following correspondence is one-to-one between the two sets:
				
				\begin{itemize}
					\item Non isomorphic $G(m,1,n)$-ramified covers $(p,f)$ of $\Complex P^1$
					of degree $mn$ 
					
					\item $Hom(\pi(\Complex P^1-Y), 
					\Integer/\Integer m \wr S_n)/ (\Integer/\Integer m \wr S_n)$
				\end{itemize}
				
			\end{lemma}
			
			\subsection{Hurwitz numbers: the algebro-geometric definition} 
			
			We now define in an algebro-geometric way Hurwitz numbers 
			for the reflection groups $\G$.
			
			We say that a $G(m,1,n)-$ramified cover $(p,f)$ of degree $mn$ 
			over $\Complex P^1$ has profile $(\overline{\lambda},n_0,n_1,\dots,n_{m-1})$ if:
			
			\begin{itemize}
				
				\item for each $\alpha=0,\dots m-1$, there are
				$n_\alpha$ other critical values $y_{\alpha,1}\dots y_{\alpha,n_\alpha}$ in 
				$\Complex P^1$,
				and the monodromy over these $y_{\alpha,1}\dots y_{\alpha,n_\alpha}$ must be 
				a $\beta_\alpha$-cycle.
				
				\item the monodromy over $\infty$ must be $\overline{\lambda}$
				
				\item the map $f\circ p$ is unramified outside of $Y= \{\infty\} \cup 
				\{y_{\alpha,i}\}_{\substack{\alpha= 0,\dots,m-1 \\
						1\leq i \leq n_\alpha}}$
			\end{itemize}
			
			The genus of $M$ can be found by using the Riemann-Hurwitz formula, and we
			do not require $M$ to be connected.
			Thus the algebro-geometric definition of $h_{n_0,n_1,\dots,n_{m-1}\overline{\lambda}}$
			is the number of non isomorphic $G(m,1,n)-$ramified cover $(p,f)$ of degree $mn$ 
			over $\Complex P^1$ having profile $(\overline{\lambda},n_0,n_1,\dots,n_{m-1})$
			divided by their weight. In a more compact formula:
			
			\begin{equation*}
				\mathbf{h}_{n_0,n_1,\dots,n_{m-1}\overline{\lambda}} = 
				\sum_{(p,f)} \frac{1}{|\text{Aut}(p,f)|}
			\end{equation*}
			where the $(p,f)$ have profile $(\overline{\lambda},n_0,n_1,\dots,n_{m-1})$
			and all distinct.
			
			\begin{proposition}\cite{Zhang}\label{Prop:HurwRamCover}
				The combinatorial definition of Hurwitz numbers $h_{n_0,n_1,\dots,n_{m-1}\overline{\lambda}}$
				and the algebro-geometric
				definitions $\mathbf{h}_{n_0,n_1,\dots,n_{m-1}\overline{\lambda}}$ are equivalent.
			\end{proposition}
			
			\section{ELSV-type formula}
			
			Given the description of the family of the cut-and-join equation, it is possible to reduce the problem of computation of the generating function $\mathcal H$ to the well-known case of usual Hurwitz numbers. The operators $\mathcal{CJ}_i$ commute for all $i = 0,\ldots,m-1$, so we can use the following strategy: first apply operators $\mathcal {CJ}_k$ to the initial condition and then try to relate the result of the action of $\mathcal {CJ}_0$ on the resulting function to the studied cases.
			\begin{proposition}
				$$e^{\sum_{k = 1}^{m-1} \beta_k \mathcal{CJ}_k} e^{p_1^{(0)}} = \prod_{\alpha = 0}^{m-1} \exp{(e^{\sum_{k = 1}^{m-1} \beta_k \xi^{k\alpha }} u_1^{(\alpha)} )}$$
			\end{proposition}
			\begin{proof}
				It is enough to compute $e^{\sum_{k = 1}^{m-1} \beta_k \xi^{k\alpha} E_{u^{(\alpha)}} }e^{u_1^{(\alpha)}}$, the result will be obtained by the product of these expressions for all $\alpha$. Denote the coefficient $\sum_{k=1}^{m-1} \beta_k \xi^{k\alpha}$ by $\gamma_\alpha$. Then we have
				\begin{gather*}
					e^{\gamma_\alpha E_{u^{(\alpha)}}}e^{u^{(\alpha)}_1} 
					= \sum_{k = 0}^\infty \frac 1{k!} \gamma_\alpha^k E_{u^{(\alpha)}}^k \sum_{l = 0}^\infty \frac 1{l!} (u_1^{(\alpha)})^l \\ =
					\sum_{k,l = 0}^\infty \frac{1}{k!l!} \gamma_\alpha^k l^k (u_1^{(\alpha)})^l = \sum_{l = 0}^\infty \frac 1{l! }(u_1^{(\alpha)})^l e^{\gamma_\alpha l} = \exp({e^{\gamma_\alpha} u_1^{(\alpha)}}).
				\end{gather*}
			\end{proof}
			
			Now, consider the action of $\mathcal{CJ}_0$. Notice, that the exponent $e^{\mathcal{CJ}_0}$ splits in the product of commuting operators, each depending on $u^{(\alpha)}$ with $\alpha $ fixed only. So, the result can be naturally written down as a product.
			
			\begin{proposition}
				The generating function 
				$$e^{\beta_0 m \mathcal{CJ}_{u^{(\alpha)}}} \exp{(e^{\sum_{k=1}^m\beta_k \xi^{k\alpha}} u_1^{(\alpha)})}   $$
				coincides with the generating function for the disconnected weighted usual simple Hurwitz numbers $e^{G^{(\alpha)}}$, where
				\begin{gather*} G^{(\alpha)} = \sum_{\begin{smallmatrix} g\ge 0, n \ge 1 \\ k_1,\ldots, k_n \in \mathbb N \end{smallmatrix}} \frac{(m\beta_0)^{2g - 2 + n + \sum_{i = 1}^n k_n }e^{\sum_{i = 1}^n k_n\sum_{k=1}^m\beta_k \xi^{k\alpha}}}{(2g - 2 + n + \sum_{i = 1}^n k_n)!} \frac{h_{g;k_1,\ldots,k_n} u_{k_1}^{(\alpha)}\ldots u_{k_n}^{(\alpha)}}{n!} .
				\end{gather*}
				
			\end{proposition}
			
			\begin{proof}
				Both functions satisfy the same evolution equation and the same initial condition. For the details see \cite{Zhu} and \cite{DK18}
			\end{proof}
			
			Usual Hurwitz numbers are known to satisfy the ELSV-formula \cite{ELSV01}: for any $g\ge 0, k_1,\ldots, k_n \in \mathbb N$ we have
			$$\mathfrak{H}_{g;k_1,\ldots,k_n} = (2g - 2 + n + \sum_{i = 1}^n k_i)! {\prod_{i = 1}^n \frac{k_i^{k_i}}{k_i!} \int_{\widebar{\mathcal M}_{g,n}} \frac {\Lambda_{g,n}}{(1 - k_1 \psi_1)\cdots (1 - k_n \psi_n)} },$$ 
			where $\widebar{\mathcal M}_{g,n}$ is the Deligne-Mumford compactification of the moduli space of genus $g$ curves with $n$ marked points, $\Lambda_{g,n}$ is the total Chern class of the dual to the Hodge vector bundle, and $\psi_i$ are the corresponding $\psi$-classes. Combining the formulae, we obtain the following:
			
			\begin{theorem}\label{Th:ELSV}
				The logarithm of the generating function $\mathcal H$ in variables $u^{(\alpha)}$ equals
				\begin{align*}
					\log{\mathcal H} = \sum_{\alpha = 0}^{m-1} G^{(\alpha)}
					=& \sum_{\alpha = 0}^{m-1} \sum_{\begin{smallmatrix} g\ge 0, n \ge 1 \\ k_1,\ldots, k_n \in \mathbb N \end{smallmatrix}} \frac{(m\beta_0)^{2g - 2 + n + \sum_{i = 1}^n k_n}e^{\sum_{i = 1}^n k_n\sum_{k=1}^m\beta_k \xi^{k\alpha}}}{n!} \\
					&\times {\prod_{i = 1}^n \frac{k_i^{k_i}}{k_i!} \int_{\widebar{\mathcal M}_{g,n}} \frac {\Lambda_{g,n}}{(1 - k_1 \psi_1)\cdots (1 - k_n \psi_n)} u_{k_1}^{(\alpha)}\ldots u_{k_n}^{(\alpha)}}.
				\end{align*}
				
			\end{theorem}

		\end{document}